\documentclass{amsart}
\usepackage{amssymb}
\usepackage{amsmath}
\usepackage{latexsym}
\usepackage{eepic}
\usepackage{epsfig}
\usepackage{graphicx}
\usepackage{pb-diagram,lamsarrow,pb-lams,amscd}
\usepackage{eufrak}
\usepackage{mathrsfs}
\usepackage[english]{babel}
\usepackage[all,cmtip]{xy}

\DeclareMathAlphabet{\mathpzc}{OT1}{pzc}{m}{it}
\newtheorem{theorem}{Theorem}[section]
\newtheorem{theorem-definition}[theorem]{Theorem-Definition}
\newtheorem{lemma-definition}[theorem]{Lemma-Definition}
\newtheorem{definition-prop}[theorem]{Proposition-Definition}

\newtheorem{prop}[theorem]{Proposition}
\newtheorem{lemma}[theorem]{Lemma}
\newtheorem{cor}[theorem]{Corollary}
\newtheorem{definition}[theorem]{Definition}

\newtheorem{example}[theorem]{Example}

\newenvironment{remark}{\vspace{4pt}\noindent\textbf{Remark.}}{\qed\vspace{4pt}}
\newenvironment{remarks}{\vspace{4pt}\noindent\textbf{Remarks.}}{\qed\vspace{4pt}}
\newtheorem{stelling?}{Stelling(?)}[section]

\newcommand{\LL}{\ensuremath{\mathbb{L}}}

\newcommand{\N}{\ensuremath{\mathbb{N}}}
\newcommand{\Z}{\ensuremath{\mathbb{Z}}}

\newcommand{\X}{\ensuremath{\mathscr{X}}}

\renewcommand{\X}{\ensuremath{\mathfrak{X}}}

\newcommand{\Spec}{\ensuremath{\mathrm{Spec}\,}}
\newcommand{\Spf}{\ensuremath{\mathrm{Spf}\,}}
\newcommand{\Sp}{\ensuremath{\mathrm{Sp}\,}}

\numberwithin{equation}{section} \hyphenpenalty=6000
\begin{document}
\title{The N\'eron component series of an abelian variety}
\author{Lars Halvard Halle}
\address{Institut f{\"u}r Algebraische Geometrie\\
Gottfried Wilhelm Leibniz Universit{\"a}t Hannover\\ Welfengarten
1
\\
30167 Hannover \\ Deutschland} \email{halle@math.uni-hannover.de}
\author[Johannes Nicaise]{Johannes Nicaise}
\address{KULeuven\\
Department of Mathematics\\ Celestijnenlaan 200B\\3001 Heverlee \\
Belgium} \email{johannes.nicaise@wis.kuleuven.be}

\thanks{The first author was partially supported by the Fund for Scientific
Research-Flanders (G.0318.06) and by DFG under grant Hu 337/6-1.
 The second author was partially supported by ANR-06-BLAN-0183 and ANR-07-JCJC-0004.
}
\begin{abstract}
We introduce the N\'eron component series of an abelian variety
$A$ over a complete discretely valued field. This is a power
series in $\Z[[T]]$, which measures the behaviour of the number of
components of the N\'eron model of $A$ under tame ramification of
the base field. If $A$ is tamely ramified, then we prove that the
N\'eron component series is rational. It has a pole at $T=1$,
whose order equals one plus the potential toric rank of $A$. This
result is a crucial ingredient of our proof of the motivic
monodromy conjecture for abelian varieties. We expect that it
extends to the wildly ramified case; we prove this if $A$ is an
elliptic curve, and if $A$ has potential purely multiplicative
reduction.
\end{abstract}

\maketitle
\section{Introduction}\label{sec-intro}
Let $K$ be a complete discretely valued field with ring of
integers $R$ and separably closed residue field $k$. We denote by
$p$ the characteristic exponent of $k$. Let $A$ be an abelian
variety over $K$. We denote its N\'eron model by $ \mathcal{A} $,
and the special fiber of $\mathcal{A}$ by $\mathcal{A}_s$. The
group $\Phi_A$ of connected components of $\mathcal{A}_s$ is a
finite abelian group, whose order we denote by $ \phi(A)$. In this
paper, we study how $\Phi_A$ and $\phi(A)$ vary under finite
extension of the base field $K$.


We denote by $\N'$ the set of strictly positive integers that are
prime to $p$, and we fix a separable closure $K^s$ of $K$. For
each element $d$ of $\N'$, there exists a unique extension $K(d)$
of $K$ in $K^s$. We define the \emph{N\'eron component series}
$S_{\phi}(A;T)$ of $A$ as
$$ S_{\phi}(A;T) = \sum_{d \in \mathbb{N}'} \phi(A\times_K K(d)) T^d \in \mathbb{Z}[[T]] $$
We will show in Theorem \ref{thm-ratseries} that, when $A$ is
tamely ramified, the series $S_{\phi}(A;T)$ is rational, and we
determine the order of its pole at $T=1$ (it equals one plus the
\emph{potential toric rank} of $A$). Recall that $A$ is tamely
ramified iff the minimal extension $L$ of $K$ where $A$ acquires
semi-abelian reduction, is a tame extension of $K$.

This result is of independent interest, but our main motivation
lies elsewhere: it is a crucial ingredient of our proof of the
\emph{motivic monodromy conjecture} for abelian varieties, a
global version of Denef and Loeser's motivic monodromy conjecture
for complex hypersurface singularities. Our results on
$S_{\phi}(A;T)$ allow us to prove the rationality of the motivic
zeta function $Z_A(\LL^{-s})$ of $A$, and to determine the order
of its (unique) pole. We've shown that this pole is equal to
Chai's base change conductor $c(A)$ of $A$ \cite{chai}. Our proof
of the conjecture will appear in \cite{HaNi-new}, see also
\cite{HaNi} for a preliminary version of the paper.

The key technical result in the present paper is Theorem
\ref{thm-main}. We denote by $t(A)$ the reductive rank of
$\mathcal{A}_s^o$. Under the assumption that $A$ is tamely
ramified,
 we prove that
$$ \phi(A \times_K K') = d^{t(A)} \phi(A) $$
for every finite extension $K'$ of $K$ such that $ d = [K':K] $ is
prime to $[L:K]$. The main tool we use in order to establish
Theorem \ref{thm-main} is the theory of \emph{rigid
uniformization}, in the sense of \cite{B-X}.

 It is natural to ask what kind of properties $ S_{\phi}(A;T) $ might have if $A$ is \emph{wildly} ramified.
We expect that Theorem \ref{thm-ratseries} (in particular, the
rationality of $S_{\phi}(A;T)$) should also hold for wildly
ramified abelian varieties. We prove this if $A$ has purely
multiplicative reduction, and if $A$ is an elliptic curve
(Proposition \ref{proposition-wildelliptic}). The general case
will be investigated in future work.

We conclude this introduction with a short overview of the paper.
In Section \ref{sec-prelim}, we gather some preliminaries on rigid
uniformization and the smooth rigid and formal sites. In Section
\ref{sec-torrank} we prove some auxiliary results on maximal
 subtori of algebraic groups in the context of N\'eron models. In
Section \ref{sec-semiabred} we discuss semi-abelian and good
reduction of semi-abelian varities. These three sections form the
technical preparation for the main results of the paper.

 Section
\ref{sec-main} is devoted to proving the key result Theorem
\ref{thm-main}, making extensive use of the results in \cite{B-X}.
In Section \ref{sec-ratcomp}, we prove the rationality of the
N\'eron component series, and we compute the order of its pole at
$T=1$ (Theorem \ref{thm-ratseries} and Proposition
\ref{proposition-wildelliptic}).

\section{Preliminaries}\label{sec-prelim}
\subsection{Notation}
We denote by $R$ a complete discrete valuation ring, with residue
field $k$ and quotient field $K$. We assume that $k$ is separably
closed. We denote by $p$ the characteristic exponent of $k$, and
by $\N'$ the ordered set of strictly positive integers that are
prime to $p$. For every $R$-scheme $X$, we denote by $\widehat{X}$
its formal completion along the special fiber.

For every field $F$, we fix a separable closure $F^s$. We denote
by $I$ the inertia group $Gal(K^s/K)$. For each element $d$ of
$\N'$, the field $K$ has a unique degree $d$ extension in $K^s$,
which we denote by $K(d)$. We denote by $R(d)$ the normalization
of $R$ in $K(d)$.

We denote by
$$(\cdot)_K:(Schemes/R)\rightarrow (Schemes/K)$$ the generic fiber
functor, and by
$$(\cdot)_s:(Schemes/R)\rightarrow (Schemes/k)$$ the special fiber
functor.
 We denote by $(stft/R)$ the category of separated
formal $R$-schemes locally topologically of finite type, by
$(Rig/K)$ the category of rigid $K$-varieties, and by
$$(\cdot)_\eta:(stft/R)\rightarrow (Rig/K)$$ the generic fiber
functor.

For every field $F$, an algebraic $F$-group is a group $F$-scheme
that is locally of finite type.
 For every smooth algebraic  $K$-group
$G$ that admits a N\'eron $lft$-model $\mathcal{G}$,
we denote by $\Phi_{G}$ the constant abelian group
$$\pi_0(\mathcal{G}_s)=\mathcal{G}_s/\mathcal{G}^o_s$$ of connected
components of $\mathcal{G}_s$, and by $\phi(G)\in
\N\cup\{\infty\}$ the cardinality of $\Phi_{G}$.

If $H$ is a group scheme over a scheme $S$, and $n$ a positive
integer, then we denote by $_n H$ the kernel of multiplication by
$n$ on $H$. Likewise, if $H$ is a constant group, we denote by $_n
H$ the subgroup of elements killed by $n$.

\subsection{Rigid uniformization and smooth
topology}\label{subsec-rigunif}  Let $A$ be an abelian
$K$-variety. We denote by $L$ the minimal extension of $K$ in
$K^s$ where $A$ acquires semi-abelian reduction, and we put
$e=[L:K]$. Recall that $A$ is tamely ramified iff $e$ is prime to
$p$ (since the residue field $k$ of $K$ is separably closed).

 A \textit{rigid uniformization}
of $A$ consists of the following data \cite[1.1]{B-X}:
\begin{itemize}
\item a semi-abelian $K$-variety $E$ that is the extension of an
abelian $K$-variety $B$ with potential good reduction by a
$K$-torus $T$:
$$0\rightarrow T\rightarrow E\rightarrow B\rightarrow 0$$

\item a lattice $M$ in $E$, of rank $\mathrm{dim}(T)$, and a
faithfully flat morphism of rigid $K$-varieties
$$E^{an}\rightarrow A^{an}$$
with kernel $M^{an}$.
 Here $(\cdot)^{an}$
denotes the rigid analytic GAGA functor.
\end{itemize}
 Beware that the group
$K$-scheme $M$ is not necessarily constant, only locally constant
for the \'etale topology.

Let $K'$ be a finite extension of $K$ in $K^s$, of degree $d$. If
we denote by $(\cdot)'$ the base change functor
$(Sch/K)\rightarrow (Sch/K')$, then
 the data
$$
\begin{CD}
0@>>> T'@>>> E'@>>> B'@>>> 0
\\ 0@>>> (M')^{an} @>>> (E')^{an} @>>> (A')^{an} @>>> 0
\end{CD}$$
define a rigid uniformization of $A'=A\times_K K'$. The objects
$M$ and $T$ split over $L$, and $B\times_K L$ has good reduction.

We denote by $R'$ the normalization of $R$ in $K'$, and by $k'$
the residue field of $R'$. The latter is a finite purely
inseparable extension of $k$. We'll consider the small rigid
smooth sites $(\Sp K)_{sm}$ and $(\Sp K')_{sm}$, and the small
formal smooth sites $(\Spf R)_{sm}$ and $(\Spf R')_{sm}$
\cite[\S\,4]{B-X}. We obtain a commutative diagram
$$\begin{CD}
(\Sp K')_{sm}@>j'>> (\Spf R')_{sm} 
\\ @Vh_K VV @VVhV 
\\ (\Sp K)_{sm}@>>j> (\Spf R)_{sm} 
\end{CD}$$
We denote by $(Ab_K),\ldots$ the category of abelian sheaves on
$(\Sp K)_{sm},\ldots$ For every commutative algebraic $K$-group
$X$, we denote by $\mathscr{F}_{X}$ the abelian sheaf on $(\Sp
K)_{sm}$ represented by $X^{an}$ \cite[3.3]{B-X}. Note that
$(h_K)^*\mathscr{F}_X$ is canonically isomorphic to
$\mathscr{F}_{X\times_K K'}$.

\begin{remark}
Beware that the functor
$$X\mapsto \mathscr{F}_X$$ from the category of commutative
algebraic $K$-groups to the category $(Ab_K)$ is \textit{not}
faithful. For instance, if $K$ has characteristic $p>1$ and
$$X=\alpha_p=\Spec K[x]/(x^p)$$ then $\mathscr{F}_X=0$.
\end{remark}

If $X$ is a smooth commutative algebraic $K$-group that admits a
N\'eron $lft$-model $\mathcal{X}$ over $R$, then the formal
completion $\widehat{\mathcal{X}}$ is a formal N\'eron model for
$X^{an}$ \cite[6.2]{formner}, and $\widehat{\mathcal{X}}$
represents the sheaf $j_* \mathscr{F}_X$ on $(\Spf R)_{sm}$.

\begin{lemma}\label{lemm-exact}
The functors \begin{eqnarray*} (h_K)^*&:&(Ab_K)\rightarrow (Ab_{K'}) \\
(h_K)_*&:&(Ab_{K'})\rightarrow (Ab_K)
\end{eqnarray*} are exact.
\end{lemma}
\begin{proof}
Since $h_K$ is smooth, the functor $(h_K)^*$ is simply the
restriction from $(\Sp K)_{sm}$ to $(\Sp K')_{sm}$, so it is
exact. It remains to show that $(h_K)_*$ is right exact. Let
$\mathscr{F}\rightarrow \mathscr{G}$ be a surjection of abelian
sheaves on $(\Sp K')_{sm}$. Let $K''$ be a finite Galois extension
of $K$ containing $K'$. The morphism $g:\Sp K''\rightarrow \Sp K$
is a covering in the rigid smooth site on $\Sp K$, so that it
suffices to show that
$$\alpha:g^*(h_K)_*\mathscr{F}\rightarrow g^*(h_K)_*\mathscr{G}$$
is surjective. If we denote by $g'$ the morphism $\Sp
K''\rightarrow \Sp K'$, then $g=h_K\circ g'$ and
 we have for every abelian sheaf $\mathscr{H}$ on $(\Sp K')_{sm}$ a canonical isomorphism
$$g^*(h_K)_*  \mathscr{H}\cong \bigoplus_{\gamma:K'\rightarrow K''} (g')^*\mathscr{H}$$
where $\gamma$ runs over the morphisms of $K$-algebras
$K'\rightarrow K''$. Surjectivity of $\alpha$ now follows from
right exactness of $(g')^*$.
\end{proof}

\subsection{The trace map}\label{subsec-trace}
We keep the notations of Section \ref{subsec-rigunif}. Let
$\mathscr{F}$ be an abelian sheaf on $(\Sp K)_{sm}$, and consider
the tautological morphism
$$\tau:\mathscr{F}\rightarrow (h_K)_* (h_K)^* \mathscr{F}$$
We will define a \textit{trace map}
$$tr:(h_K)_* (h_K)^* \mathscr{F}\rightarrow \mathscr{F}$$ such
that the composition $tr\circ \tau$ is multiplication by
$d=[K':K]$.

Let $K''$ be a Galois extension of $K$ that contains $K'$. The
morphism $g:\Sp K''\rightarrow \Sp K$ is a covering in the rigid
smooth site on $\Sp K$, and we have a canonical isomorphism
$$g^*(h_K)_* (h_K)^* \mathscr{F}\cong \bigoplus_{\gamma:K'\rightarrow K''} g^*\mathscr{F}$$
where $\gamma$ runs over the morphisms of $K$-algebras
$K'\rightarrow K''$. Consider the morphism
$$tr_{K''}:\bigoplus_{\gamma:K'\rightarrow K''}
g^*\mathscr{F}\rightarrow g^*\mathscr{F}:(\alpha_{\gamma})\mapsto
\sum_{\gamma}\alpha_{\gamma}$$ It is invariant under the action of
the Galois group $Gal(K''/K)$. This implies that $tr_{K''}$
satisfies the gluing conditions w.r.t. the covering $g$, so that
$tr_{K''}$ descends to a morphism of sheaves
$$tr:(h_K)_* (h_K)^* \mathscr{F}\rightarrow \mathscr{F}$$ on $(\Sp
K)_{sm}$. The composition $tr\circ \tau$ is multiplication by $d$,
since this holds after base change to $K''$.

\section{Toric rank}\label{sec-torrank}
\subsection{Subtori of algebraic groups}
We recall some results on subtori of algebraic groups. We focus on
the case of smooth commutative algebraic groups over a field.

\begin{definition}
Let $G$ be an algebraic group over a field $F$. A subtorus $T$ of
$G$ is called a maximal subtorus of $G$ if, for some algebraically
closed field $F'$ containing $F$, the algebraic group $G\times_F
F'$ does not admit any subtorus that is strictly larger than
$T\times_F F'$.
\end{definition}
If this property holds for some algebraically closed extension
$F'$, then it holds for all algebraically closed overfields of $F$
\cite[XII.1.2]{sga3.2}.

\begin{prop}\label{prop-maxtor}
Let $G$ be a smooth commutative algebraic group over a field $F$.
\begin{enumerate}
\item The algebraic group $G$ admits a unique maximal subtorus.
\item If $T$ is an $F$-torus, then every morphism of algebraic
$F$-groups $T\rightarrow G$ factors through the maximal subtorus
of $G$.
\end{enumerate}
\end{prop}
\begin{proof}
(1) Existence follows from \cite[XIV.1.1]{sga3.2}, uniqueness from
\cite[XII.7.1(b)]{sga3.2} and commutativity of $G$.

(2) Denote by $S$ the maximal subtorus of $G$. We have to show
that all morphisms of algebraic $F$-groups
$$f:T\rightarrow G/S$$ are trivial. We may assume that $F$ is
algebraically closed, and that $G$ is connected. Since $S$ is the
maximal subtorus of $G$, we know by the Chevalley decomposition
\cite{conrad-chevalley} that $G/S$ is the extension of an abelian
variety $A$ by a unipotent $F$-group $U$. Since $T$ does not admit
any non-trivial morphism to $U$ \cite[XVII.2.4]{sga3.2} or to $A$
\cite[2.3]{conrad-chevalley}, we see that $f$ is trivial.
\end{proof}

\begin{definition}
Let $G$ be a smooth commutative algebraic group over a field $F$.
The reductive rank $\rho(G)$ of $G$ is the dimension of the
maximal subtorus of $G$.
\end{definition}

\begin{prop}\label{prop-redzero}
Let $G$ be a smooth commutative algebraic group over a field $F$.
If $T$ is the maximal subtorus of $G$, then
$$\rho(G/T)=0$$
and there exist no non-trivial morphisms from an $F$-torus to
$G/T$.
\end{prop}
\begin{proof}
By Proposition \ref{prop-maxtor}, it suffices to show that
$\rho(G/T)=0$. Let $S$ be a subtorus of $G/T$. Then we have a
commutative diagram with exact rows and columns
$$\begin{CD}
0@>>> T@>>> G@>>> G/T @>>> 0
\\ @. @A= AA @AAA @AAA @.
\\ 0@>>> T@>>> G\times_{G/T} S@>>> S @>>> 0
\\ @. @. @AAA @AAA @.
\\ @. @. 0 @. 0 @.
\end{CD}$$ Since $G\times_{G/T} S$ is an extension of two
$F$-tori, it is again an $F$-torus (it is of multiplicative type
by \cite[IX.8.2]{sga3.2}, and it is smooth and connected). Since
$T$ is the maximal subtorus of $G$, we have that $T\rightarrow
G\times_{G/T}S$ is an isomorphism, so that $S$ must be trivial.
\end{proof}

\begin{prop}\label{prop-maxsplittor}
If $G$ is a smooth commutative algebraic group over a field $F$,
then $G$ admits a unique split subtorus $T_{sp}$ such that for
every split $F$-torus $S$ and every morphism of algebraic
$F$-groups $f:S\rightarrow G$, the morphism $f$ factors through
$T_{sp}$.

If $G$ is a torus with character module $X(G)$, then the dimension
of $T_{sp}$ is equal to the rank of the free $\Z$-module
$X(G)^{Gal(F^s/F)}$.
\end{prop}
\begin{proof}
By Proposition \ref{prop-maxtor}, we may assume that $G$ is a
torus. Let $F'$ be a splitting field of $G$, denote by $X(G)$ the
character module of $G$, and consider the trace map
$$tr:X(G)\rightarrow X(G)^{Gal(F'/F)}:x\mapsto \sum_{\gamma \in
Gal(F'/F)}\gamma \cdot x$$
 By the duality between tori
and their character modules, it is clear that $T_{sp}$ is the
torus corresponding to the character module $X(G)/ker(tr)$
(cf.~\cite[1.3]{N-X}). Since the restriction of $tr$ to
$$X(G)^{Gal(F'/F)}=X(G)^{Gal(F^s/F)}$$ is multiplication by $[F':F]$,
we see that
$$\mathrm{dim}(T_{sp})=\mathrm{rank}_{\Z}(X(G)/ker(tr))=\mathrm{rank}_{\Z}(X(G)^{Gal(F^s/F)})$$
\end{proof}
\begin{definition}
With the notations of Proposition \ref{prop-maxsplittor}, we call
$T_{sp}$ the maximal split subtorus of $G$.
\end{definition}
An argument similar to the proof of Proposition \ref{prop-redzero}
shows that the maximal split subtorus of $G/T_{sp}$ is trivial.
\subsection{Toric rank of a semi-abelian variety}
\begin{definition}
Let $G$ be a semi-abelian $K$-variety, with N\'eron $lft$-model
$\mathcal{G}$. We define the toric rank $t(G)$ of $G$ by
$$t(G)=\rho(\mathcal{G}_s^o)$$
\end{definition}

\begin{lemma}\label{lemma-tor}
Let
$$f:\mathcal{H}\rightarrow \mathcal{G}$$ be a morphism of smooth
group $R$-schemes such that $f_K$ is injective. Then
$$\rho(\mathcal{G}_s^o)\geq \rho(\mathcal{H}_s^o)$$
\end{lemma}
\begin{proof}
Denote by $S$ and $T$ the maximal subtori of $\mathcal{G}_s^o$,
resp. $\mathcal{H}_s^o$.
 We consider the commutative
diagram
$$\begin{CD}
_\ell T(k)@>(*)>> _\ell \mathcal{H}_s^o(k) @>>> _\ell
\mathcal{G}^o_s(k)
\\ @. @AAA @AAA
\\ @. _\ell \mathcal{H}^o(R) @>(*)>>
_\ell \mathcal{G}^o(R)
\end{CD}$$
The vertical arrows are bijections, by \cite[7.3.3]{neron}, and
the arrows marked by (*) are injections. It follows that
$$_\ell T(k)\rightarrow\,  _\ell \mathcal{G}_s^o(k)$$ is injective. The morphism $T\rightarrow \mathcal{G}^o_s$ factors through $S$,
by Proposition \ref{prop-maxtor}(2). Injectivity of
$$_\ell T(k)\rightarrow\, _\ell S(k)$$ implies that
$\mathrm{dim}(S)\geq \mathrm{dim}(T)$, since the number of
elements of $_\ell\mathbb{G}^d_{m}(k)$ equals $\ell^d$, for every
integer $d\geq 0$.
\end{proof}

\begin{prop}\label{prop-torrank}
Let $G$ be a semi-abelian $K$-variety. For every finite extension
$K'$ of $K$, we have
$$t(G)\leq t(G\times_K K')$$
with equality if $G$ has semi-abelian reduction.
\end{prop}
\begin{proof}
We denote by $R'$ the normalization of $R$ in $K'$, by $k'$ its
residue field, and by $\mathcal{G}$ and $\mathcal{G}'$ the N\'eron
$lft$-models of $G$ and $G'$. We consider the unique $R'$-morphism
$$\mathcal{G}\times_R R'\rightarrow \mathcal{G}'$$ extending the
canonical isomorphism between the generic fibers. Applying Lemma
\ref{lemma-tor} to this morphism, we see that
$$t(G)\leq t(G\times_K K')$$
If $G$ has semi-abelian reduction, then
$$\mathcal{G}_s^o\times_k k'\rightarrow (\mathcal{G}')^o_s$$
is an isomorphism \cite[IX.3.1(e)]{sga7a}, so that
$$t(G)= t(G\times_K K')$$
\end{proof}
\begin{definition}
Let $G$ be a semi-abelian $K$-variety, and let $L$ be a finite
separable extension of $K$ such that $G\times_K L$ has
semi-abelian reduction. We define the potential toric rank
$t_{pot}(G)$ of $G$ by
$$t_{pot}(G)=t(G\times_K L)$$
\end{definition}
By Proposition \ref{prop-torrank}, this definition does not depend
on $L$, and we have
$$t_{pot}(G)=\max\{t(G\times_K K')\,|\,K'\mbox{ a finite extension
of }K\}$$ We say that $G$ has \emph{purely multiplicative
reduction} if $t_{pot}(G)$ is equal to the dimension of $G$, i.e.,
if the identity component of the special fiber of the N\'eron
model of $G\times_K L$ is a torus.

\begin{remark}
The existence of $L$ (i.e., the potential semi-abelian reduction
of $G$) is well-known. It is easily deduced from the semi-abelian
reduction theorem for abelian varieties \cite[IX.3.6]{sga7a}; see
the implication $(2)\Rightarrow (1)$ in Proposition
\ref{prop-semiabred} below.
\end{remark}

\begin{definition}
Let $G$ be a semi-abelian $K$-variety, with toric part $G_{tor}$
and abelian part $G_{ab}$. We say that $G$ has good reduction if
$G_{tor}$ and $G_{ab}$ have good reduction. We say that $G$ has
potential good reduction if there exists a finite separable
extension $K'$ of $K$ such that $G\times_K K'$ has good reduction.
\end{definition}
Note that every algebraic $K$-torus has potential good reduction.

\begin{prop}\label{prop-uniftor}
Let $G$ be a semi-abelian $K$-variety with potential good
reduction. If we denote by $T_{sp}$ the maximal split subtorus of
$G$, then
\begin{eqnarray*}t(G)&=&\mathrm{dim}(T_{sp})
\\ t_{pot}(G)&=&\rho(G)
\end{eqnarray*}
\end{prop}
\begin{proof}
The second equality follows from the first, passing to a splitting
field of the maximal subtorus of $G$. So let us prove the first
equality. We denote by $\mathcal{G}$ the N\'eron $lft$-model of
$G$.

\textit{Case 1: $T_{sp}$ is trivial, and $G$ is a torus.} By
\cite[2.3.1]{ananth}, $\mathcal{G}^o$ is affine. Consider a
morphism
$$f_s:\mathbb{G}_{m,k}\rightarrow \mathcal{G}_s^o$$
By \cite[IX.7.3]{sga3.2}, the morphism $f_s$ lifts uniquely to a
morphism of group $R$-schemes $f:\mathbb{G}_{m,R}\rightarrow
\mathcal{G}^o$. Passing to the generic fiber, we find a morphism
of algebraic $K$-groups $f_K:\mathbb{G}_{m,K}\rightarrow G$. Since
$T_{sp}$ is a point, $f_K$, and hence $f_s$, must be trivial, so
that $t(G)=0$.

%
\textit{Case 2: $T_{sp}$ is trivial.} We have to show that
$t(G)=0$. We denote by
$$0\rightarrow G_{tor}\rightarrow G\rightarrow G_{ab}\rightarrow
0$$ the Chevalley decomposition of $G$, and by
$$\mathcal{G}_{tor}^o\rightarrow \mathcal{G}^o\rightarrow
\mathcal{G}_{ab}^o$$ the induced sequence on identity components
of N\'eron $lft$-models. Taking formal completions and passing to
the generic fiber, we find a sequence of rigid $K$-groups
$$(\widehat{\mathcal{G}}^o_{tor})_\eta\rightarrow (\widehat{\mathcal{G}}^o)_\eta\rightarrow
(\widehat{\mathcal{G}}^o_{ab})_\eta$$

We denote by $\mathbb{G}^{rig}_{m,K}$ the generic fiber of the
formal group $R$-scheme $\widehat{\mathbb{G}}_{m,R}$. It is a
rigid $K$-group, and it coincides with the unit circle in the
rigid analytification $(\mathbb{G}_{m,K})^{an}$.

Assume that there exists a non-trivial morphism of algebraic
$k$-groups
$$f_s:\mathbb{G}_{m,k}\rightarrow \mathcal{G}_s^o$$
By the infinitesimal lifting property for tori
\cite[IX.3.6]{sga3.2} we know that this morphism lifts uniquely to
a morphism of formal group $R$-schemes
$f:\widehat{\mathbb{G}}_{m,R}\rightarrow \widehat{\mathcal{G}}^o$.
Passing to the generic fiber, we find a morphism of rigid
$K$-groups $f_\eta:\mathbb{G}_{m,K}^{rig}\rightarrow
(\widehat{\mathcal{G}}^o)_\eta$.

We consider the composed morphism
$g_\eta:\mathbb{G}_{m,K}^{rig}\rightarrow
(\widehat{\mathcal{G}}^o_{ab})_\eta$. By the universal property of
the formal N\'eron model \cite[1.1]{formner}, $g_\eta$ extends
uniquely to a morphism of formal group $R$-schemes
$g:\widehat{\mathbb{G}}_{m,R}\rightarrow
\widehat{\mathcal{G}}^o_{ab}$. Passing to the special fiber, we
obtain a morphism of algebraic $k$-groups
$g_s:\mathbb{G}_{m,k}\rightarrow (\mathcal{G}_{ab})^o_s$. By
Proposition \ref{prop-torrank}, the fact that $G_{ab}$ has
potential good reduction implies that $t(G_{ab})=0$. Hence, $g_s$
is trivial. It follows from \cite[IX.3.5]{sga3.2} that $g$ is
trivial, so that the image of $f_\eta$ is contained in
$(\widehat{\mathcal{G}}_{tor}^o)_\eta$. But $t(G_{tor})=0$ by Case
1, so repeating the above argument we see that $f_\eta$ is
trivial, so that $f_s$ is trivial. Hence, $t(G)=0$.

 \textit{Case 3: general case.}
We denote by $\mathcal{T}$ the N\'eron $lft$-model of $T_{sp}$. We
consider the unique $R$-morphism
$$f:\mathcal{T}^o\rightarrow \mathcal{G}^o$$
extending $T_{sp}\rightarrow G$. By Lemma \ref{lemma-tor}, we have
$$t(G)\geq \mathrm{dim}(T_{sp})$$
It remains to prove the converse inequality.

If we put $H=G/T_{sp}$, then $t(H)=0$ by Case 2. Copying the proof
of Case 2, we see that the image of any morphism of rigid
$K$-groups $\mathbb{G}_{m,K}^{rig}\rightarrow G^{an}$ is contained
in $T^{an}_{sp}$, and that any morphism of algebraic $k$-groups
$\mathbb{G}_{m,k}\rightarrow \mathcal{G}_s^o$ lifts to a morphism
$\mathbb{G}_{m,k}\rightarrow \mathcal{T}_s^o$. Hence,
$$t(G)\leq \mathrm{dim}(T_{sp})$$
\end{proof}

\begin{prop}\label{prop-uniftor2}
Let $A$ be an abelian $K$-variety. We adopt the notation of
Section \ref{subsec-rigunif}. We consider $M$ as a discrete
$I$-module, and we denote by $T_{sp}$ the maximal split subtorus
of $T$. Then we have
$$\begin{array}{rcccl}
\mathrm{rank}_{\Z}(M)&=&\mathrm{dim}(T)&=&t_{pot}(A)
\\ \mathrm{rank}_{\Z}(M^I)&=&\mathrm{dim}(T_{sp})&=&t(A)
\end{array}$$
\end{prop}
\begin{proof}
Since $T$ and $M$ split over $L$, the first statement follows from
the second.
 If we denote by
$\mathcal{E}$ the N\'eron $lft$-model of $E$, then
$\mathcal{E}_s^o$ is isomorphic to $\mathcal{A}_s^o$
\cite[2.3]{B-X}. It follows from Proposition \ref{prop-uniftor}
that
$$t(E)=\mathrm{dim}(T_{sp})$$
so we find
$$t(A)=\mathrm{dim}(T_{sp})$$

Since $T_{sp}$ is split, we know that $R^1 j_*
\mathscr{F}_{T_{sp}}=0$ \cite[4.2]{B-X}.
 By \cite[4.4+9+11+12]{B-X}
we have exact sequences
$$\begin{CD}
0@>>> \Phi_{T_{sp}}@>>> \Phi_E @>>> \Phi_{E/T_{sp}}
\\ 0@>>> M^I @>>> \Phi_E @>>> \Phi_A
\end{CD}$$
But $\Phi_{E/T_{sp}}$ and $\Phi_A$ are finite \cite[10.2.1]{neron}
so that
$$\mathrm{rank}_{\Z}(M^I)=\mathrm{rank}_{\Z}(\Phi_E)=\mathrm{rank}_{\Z}(\Phi_{T_{sp}})=\mathrm{dim}(T_{sp})$$
where the last equality follows from the description of the
N\'eron $lft$-model of $\mathbb{G}_{m,K}$ in \cite[10.1.5]{neron}.
\end{proof}

\section{Semi-abelian reduction of semi-abelian varieties}\label{sec-semiabred}
\begin{prop}\label{prop-semiabred}
Let $G$ be a semi-abelian $K$-variety, with toric part $G_{tor}$
and abelian part $G_{ab}$. The following are equivalent:
\begin{enumerate}
\item $G$ has semi-abelian reduction \item $G_{ab}$ has
semi-abelian reduction, and $G_{tor}$ is split \item the action of
$I$ on $T_\ell G$ is unipotent.
\end{enumerate}
\end{prop}
\begin{proof}
The sequence of Tate modules
$$0\rightarrow T_\ell G_{tor}\rightarrow T_\ell G\rightarrow
T_\ell G_{ab}\rightarrow 0$$ is exact. Points (2) and (3) are
equivalent if $G_{tor}$ or $G_{ab}$ is trivial, by
  \cite[IX.3.8]{sga7a} and the canonical isomorphism of $I$-modules
$$T_\ell(G_{tor})\cong X(G_{tor})^{\vee}\otimes_{\Z} \Z_\ell$$
where $X(G_{tor})^{\vee}$ is the cocharacter module of $G_{tor}$.
 It follows that (2) and (3) are equivalent for arbitrary $G$.

If (2) holds, then by \cite[10.1.7]{neron}, the sequence of
identity components of N\'eron $lft$-models
$$0\rightarrow \mathcal{G}^o_{tor} \rightarrow
\mathcal{G}^o\rightarrow \mathcal{G}^o_{ab}\rightarrow 0$$ is
exact, so that $\mathcal{G}_s^o$ is semi-abelian.

So it suffices to prove the implication $(1)\Rightarrow (2)$. We
denote by $g$, $g_{tor}$ and $g_{ab}$ the dimensions of $G$,
$G_{tor}$ and $G_{ab}$, respectively.

Assume that $G_{tor}$ is split. Then, again by
\cite[10.1.7]{neron}, the sequence of identity components of
N\'eron $lft$-models
$$0\rightarrow \mathcal{G}^o_{tor} \rightarrow
\mathcal{G}^o\rightarrow \mathcal{G}^o_{ab}\rightarrow 0$$ is
exact. Since $(\mathcal{G}_{tor})^o_s$ is a torus and
$\mathcal{G}^o_s$ semi-abelian, we see that
$(\mathcal{G}_{ab})^o_s$ is semi-abelian.

Hence, it is enough to show that $G_{tor}$ is split, or
equivalently, that $t(G_{tor})=g_{tor}$ (Proposition
\ref{prop-uniftor}). We denote by $T_{tor}$, $T$ and $T_{ab}$ the
maximal subtori of $(\mathcal{G}_{tor})^o_s$, $\mathcal{G}_s^o$
and $(\mathcal{G}_{ab})^o_s$, respectively.

 Consider the commutative diagram
$$\begin{CD}
0@>>> _\ell \mathcal{G}_{tor}^o(R)@>>>  _\ell \mathcal{G}^o(R)
@>>>  _\ell \mathcal{G}_{ab}^o(R)
\\ @. @VVV @VVV @VVV
\\ @.  _\ell (\mathcal{G}_{tor})_s^o(k)@>>>  _\ell \mathcal{G}^o_s(k)
@>>>  _\ell (\mathcal{G}_{ab})_s^o(k)
\\ @. @A f AA @AAA @AAA
\\ 0@>>> _\ell T_{tor}(k) @>>> _\ell T(k) @>>> _\ell T_{ab}(k)
\end{CD}$$
The first row is exact, the upper vertical arrows are bijections
by
 \cite[7.3.3]{neron}, and the lower vertical arrows are injective.
 Moreover, since $(\mathcal{G}_{tor})_s^o$ is affine \cite[2.3.1]{ananth}, we know that
 $U:=(\mathcal{G}_{tor})_s^o/T_{tor}$ is unipotent \cite[XVII.7.2.1]{sga3.2}, so
 that $_\ell U=0$ and
 $f$ is bijective. It follows that the third row is exact, too.
 Looking at the cardinality of its entries, we see that
$$\ell^{\mathrm{dim}(T)-\mathrm{dim}(T_{tor})}\leq
\ell^{\mathrm{dim}(T_{ab})}$$ and, hence, that
\begin{equation}\label{eq-ineq}
t(G)-t(G_{tor})\leq t(G_{ab})\end{equation}

Let $K'$ be a finite separable extension of $K$ such that
$G_{tor}\times_K K'$ splits, and $G_{ab}\times_K K'$ has
semi-abelian reduction. If we denote by $\mathcal{G}'_{tor}$,
$\mathcal{G}'$ and $\mathcal{G}'_{ab}$ the N\'eron models of
$G_{tor} \times_K K'$, $G \times_K K'$ and $G_{ab} \times_K K'$, respectively, then the sequence
$$0\rightarrow (\mathcal{G}_{tor}')_s^o\rightarrow
(\mathcal{G}')_s^o\rightarrow (\mathcal{G}'_{ab})_s^o \rightarrow
0$$ is exact \cite[10.1.7]{neron}. This implies that
\begin{equation}\label{eq-poteq}
t_{pot}(G_{ab})=t_{pot}(G)-t_{pot}(G_{tor})\end{equation} On the
other hand, by (\ref{eq-ineq}) and Proposition \ref{prop-torrank},
we find
\begin{eqnarray*}
t_{pot}(G)-t_{pot}(G_{tor})&=&t(G)-t_{pot}(G_{tor})
\\ &\leq& t(G)-t(G_{tor})
\\ &\leq& t(G_{ab})
\\ &\leq& t_{pot}(G_{ab})
\end{eqnarray*}
Since the first and last term of the inequality are equal by
(\ref{eq-poteq}), we may conclude that
$$t(G_{tor})=t_{pot}(G_{tor})=g_{tor}$$
\end{proof}
\begin{cor}\label{cor-tornochange} Let $G$ be a semi-abelian $K$-variety with potential good
reduction, or an abelian $K$-variety. Denote by $e$ the degree of
the minimal extension $L$ of $K$ where $G$ acquires semi-abelian
reduction. If $K'$ is a finite separable extension of $K$ such
that $[K':K]$ is prime to $e$, then
$$t(G)=t(G\times_K K')$$
\end{cor}
\begin{proof}
Assume that $G$ is an abelian $K$-variety. If $E$ is the
semi-abelian variety appearing in the rigid uniformization of $G$,
then $E$ has potential good reduction, and $t(A)=t(E)$ by
\cite[2.3]{B-X}. Hence, it suffices to consider the case where $G$
is a semi-abelian $K$-variety with potential good reduction.

Denote by $T$ the maximal subtorus of $G$, and by $T_{sp}$ the
maximal split subtorus of $T$. By Proposition \ref{prop-uniftor},
it suffices to show that $T'_{sp}:=T_{sp}\times_K K'$ is the
maximal split subtorus of $T'=T\times_K K'$. By Proposition
\ref{prop-maxsplittor}, it is enough to show that
$$X(T)^{I}=X(T)^{I'}$$ where $X(T)$ is the character module of
$T$, and $I'=Gal(K^s/K')$.

We choose an embedding of $L$ in $K^s$. We know that $T$ splits
over $L$, by Proposition \ref{prop-semiabred}. Hence, $Gal(K^s/L)$
acts trivially on the character module $X(T)$ of $T$. Since
$[L:K]$ is prime to $[K':K]$, we know that the restriction
morphism
$$Gal(K'L/K')\rightarrow Gal(L/K)$$ is an isomorphism, so that
$$X(T)^{I}=X(T)^{I'}$$
\end{proof}

\begin{prop}\label{prop-goodred}
Let $G$ be a semi-abelian $K$-variety, with toric part $G_{tor}$
and abelian part $G_{ab}$.

The following are equivalent:
\begin{enumerate}
\item $G$ has good reduction \item $G$ has semi-abelian reduction,
and  $G_{ab}$ has potential good reduction \item the action of $I$
on $T_\ell G$ is trivial.
\end{enumerate}

Moreover, $G$ has potential good reduction iff $G_{ab}$ has
potential good reduction.
\end{prop}
\begin{proof}
It follows immediately from the definition that $G$ has potential
good reduction iff $G_{ab}$ has potential good reduction.
 The sequence of Tate modules
$$0\rightarrow T_\ell G_{tor}\rightarrow T_\ell G\rightarrow
T_\ell G_{ab}\rightarrow 0$$ is exact, so that $I$ acts trivially
on $T_\ell G_{ab}$ and $T_\ell G_{tor}$ if $I$ acts trivially on
$T_\ell G$. Then $G_{ab}$ has good reduction, by the criterion of
N\'eron-Ogg-Shafarevich \cite[IX.2.2.9]{sga7a}, and $G_{tor}$ is
split, by the canonical isomorphism of $I$-modules
$$T_\ell(G_{tor})\cong X(G_{tor})^{\vee}\otimes_{\Z} \Z_\ell$$
where $X(G_{tor})^{\vee}$ is the cocharacter module of $G_{tor}$.
This proves $(3)\Rightarrow (2)$. The implication $(2)\Rightarrow
(1)$ follows from Proposition \ref{prop-semiabred}.

It remains to show that $(1)\Rightarrow (3)$.
 This can be proven
in the same way as \cite[IX.2.2.9]{sga7a}, namely, by noting that
the free $\Z_\ell$-modules
$$T_\ell \mathcal{G}_s^o\cong T_\ell \mathcal{G}^o(R)=(T_\ell G)^I$$ and $T_\ell G$ both have rank
$2g_{ab}+g_{tor}$, with $g_{ab}$ and $g_{tor}$ the dimensions of
$G_{ab}$, resp. $G_{tor}$.
\end{proof}

\begin{prop}\label{prop-compgr}
If $G$ is a semi-abelian $K$-variety, then there exists a
canonical isomorphism
$$(\Phi_G)_\ell\cong H^1(I,T_\ell G)_\mathrm{tors}$$
where $(\Phi_G)_\ell$ denotes the $\ell$-primary part of $\Phi_G$
(the subgroup of elements killed by a power of $\ell$), and
$H^1(I,T_\ell G)_\mathrm{tors}$ the torsion part of $H^1(I,T_\ell
G)$.
\end{prop}
\begin{proof}
This is a generalization of \cite[IX.11.3.8]{sga7a}, and the proof
remains valid.
\end{proof}

\section{Behaviour of the component group under
ramification}\label{sec-main}

\begin{lemma}\label{lemma-linalg}
Consider a commutative diagram of abelian groups
$$\begin{CD}
0@>>> M_1 @>>> N_1 @>f>> P_1 @>>> Q_1@>>> 0
\\ @. @V\alpha V\wr V @V\beta VV @VV\gamma V @V\wr V\delta V @.
\\ 0@>>> M_2 @>>> N_2 @>>g> P_2 @>>> Q_2@>>> 0
\end{CD}$$
 with exact rows, where $\alpha$ and $\delta$ are isomorphisms.
Then $f$ induces an isomorphism $ker(\beta)\cong ker(\gamma)$, and
$g$ induces an isomorphism $coker(\beta)\cong coker(\gamma)$.
\end{lemma}
\begin{proof}
Note that $M_1\cap ker(\beta)=\{0\}$ and that $M_2\subset
im(\beta)$. Dividing the first row by $M_1$ and the second by
$M_2$, we may assume that $M_1=M_2=0$. Now the result follows from
an easy diagram chase.
\end{proof}

 In \cite[4.7]{B-X}, Bosch and Xarles
constructed the identity component $\mathscr{F}^o$ for an
arbitrary abelian sheaf $\mathscr{F}$ on $(\Spf R)_{sm}$.  If the
sheaf $\mathscr{F}$ is representable by a formal group $R$-scheme
$\X$, then the identity component $\X^o$ represents the sheaf
$\mathscr{F}^o$. The component sheaf of $\mathscr{F}$ is defined
by
$$\Phi_{\mathscr{F}}=\mathscr{F}/\mathscr{F}^o$$
If $\mathcal{X}$ is a smooth commutative group $R$-scheme and
$\mathcal{F}$ is represented by the formal completion
$\widehat{\mathcal{X}}$, then $\Phi_{\mathscr{F}}$ is the constant
sheaf on $(\Spf R)_{sm}$ associated to the group
$\pi_0(\mathcal{X}_s)=\mathcal{X}_s/\mathcal{X}_s^o$ of connected
components of $\mathcal{X}_s$. To see this, note that the obvious
morphism of sheaves $\mathscr{F}\rightarrow \pi_0(\mathcal{X}_s)$
is surjective (by smoothness of $\mathcal{X}$) and that its kernel
is precisely $\mathscr{F}^o$.
\begin{prop}\label{prop-weilres}
Let $G$ be a semi-abelian $K$-variety, with N\'eron $lft$-model
$\mathcal{G}$. Let $K'$ be a finite separable extension of $K$,
and denote by $\mathcal{G}'$ the N\'eron $lft$-model of
$G'=G\times_K K'$. Using the notation from Section
\ref{subsec-rigunif}, there exists a canonical isomorphism
$$h^*\Phi_{j_*(h_K)_*\mathscr{F}_{G'}}\cong \Phi_{j'_*\mathscr{F}_{G'}}$$
and these abelian sheaves on $(\Spf R')_{sm}$ are canonically
isomorphic to the constant sheaf associated to the group
$\Phi_{G'}$.
\end{prop}
\begin{proof}
The sheaf $j'_*\mathscr{F}_{G'}$ is represented by the smooth
formal $R'$-scheme $\widehat{\mathcal{G}'}$, so that
$\Phi_{j'_*\mathscr{F}_{G'}}$ is canonically isomorphic to the
constant sheaf $\Phi_{G'}$.

Since $(\mathcal{G}')^o$ is quasi-projective \cite[6.4.1]{neron},
the Weil restriction
$$\prod_{R'/R}\mathcal{G}'$$
of $\mathcal{G}'$ with respect to $R\rightarrow R'$ is
representable by a separated smooth group $R$-scheme $\mathcal{W}$
\cite[7.6.4+5]{neron}. Since Weil restriction commutes with base
change, the generic fiber $W:=\mathcal{W}_K$  is the Weil
restriction of $G'$ to $K$. It is obvious that $\mathcal{W}$ is a
N\'eron $lft$-model for $W$.

 By
\cite[1.19]{bertapelle}, there is a canonical isomorphism
$$\mathscr{F}_{W}\cong (h_K)_* \mathscr{F}_{G'}$$
 By the
same arguments as above, the abelian sheaf
$$h^*\Phi_{j_*(h_K)_*\mathscr{F}_{G'}}$$ on $(\Spf R')_{sm}$ is
canonically isomorphic to the constant sheaf $\Phi_{W}$.
 Hence, it suffices to construct a canonical isomorphism between
 $\Phi_W$ and $\Phi_{G'}$.

For this, we can copy the last part of the proof of Theorem 1 in
\cite{ELL}, where the authors construct a smooth surjective
morphism of algebraic $k'$-groups $$\mathcal{W}_s\times_k
k'\rightarrow \mathcal{G}'_s$$ with connected kernel (in
\cite{ELL}, it is assumed that $G$ is an abelian variety, but the
proof is also valid for semi-abelian varieties).
\end{proof}

\begin{cor}\label{cor-killd}
We keep the notations of Proposition \ref{prop-weilres}. The
natural morphism of group $R'$-schemes
$$\mathcal{G}\times_R R'\rightarrow \mathcal{G}'$$
induces a morphism of abelian groups
$$\alpha:\Phi_G\rightarrow \Phi_{G'}$$
whose kernel is killed by $d=[K':K]$.
\end{cor}
\begin{proof}
We consider the sequence
$$\begin{CD}
\mathscr{F}_{G}@>\tau >> (h_K)_*(h_K)^*\mathscr{F}_{G}@>tr>>
\mathscr{F}_{G}
\end{CD}$$
where $\tau$ is the tautological morphism, and $tr$ the trace map.
We know that $tr\circ \tau$ is multiplication by $d$. Applying the
functor $j_*$ and passing to component groups, we obtain a
sequence
$$\begin{CD}
\Phi_{G}@>>> \Phi_{j_*(h_K)_*(h_K)^*\mathscr{F}_{G}}@>>> \Phi_{G}
\end{CD}$$
of abelian sheaves on $(\Spf R)_{sm}$. Applying $h^*$, and using
Proposition \ref{prop-weilres}, this yields a sequence of constant
groups
$$\begin{CD}
\Phi_{G}@>\alpha >> \Phi_{G'}@>\beta>> \Phi_{G}
\end{CD}$$
such that $\beta\circ \alpha$ is multiplication by $d$. The result
follows.
\end{proof}

In the case where $G$ is an abelian variety, Corollary
\ref{cor-killd} is equivalent to Theorem 1 in \cite{ELL}.

\begin{cor}\label{cor-potgood}
Let $G$ be a semi-abelian $K$-variety. Assume that $G$ acquires
good reduction over a finite separable field extension $K'$ of
$K$. Then the torsion part of $\Phi_{G}$ is killed by $d=[K':K]$.
\end{cor}
\begin{proof}
We put $G'=G\times_K K'$. By Corollary \ref{cor-killd}, it
suffices to show that $\Phi_{G'}$ has no torsion. Denote by
$$0\rightarrow G'_{tor}\rightarrow G'\rightarrow G'_{ab}\rightarrow
0$$ the Chevalley decomposition of $G'$. Then $G'_{tor}$ is a
split torus, and $G'_{ab}$ has good reduction. It follows from
\cite[4.11]{B-X} that the sequence
$$0\rightarrow \Phi_{G'_{tor}}\rightarrow \Phi_{G'}\rightarrow
\Phi_{G'_{ab}}\rightarrow 0$$ is exact. But $\Phi_{G'_{ab}}=0$,
and $\Phi_{G'_{tor}}$ is $\Z$-free.
\end{proof}

\begin{prop}\label{prop-potgood}
Let $G$ be a semi-abelian variety, and let $L$ be a finite
separable extension of $K$ such that $G$ acquires good reduction
over $L$. Let $K'$ be a finite separable extension of $K$ such
that $d=[K':K]$ is prime to $e=[L:K]$. Put $G'=G\times_K K'$, and
consider the morphism
$$\alpha:\Phi_{G}\rightarrow \Phi_{G'}$$ from Corollary
\ref{cor-killd}. Then the following hold:
\begin{enumerate}
\item the morphism $\alpha$ is injective \item if $L$ is a tame
extension of $K$, and $t(G)=0$, then $\alpha$ is an isomorphism
\item if $G$ is a torus and $t(G)=0$, then $\alpha$ is an
isomorphism.
\end{enumerate}
\end{prop}
\begin{proof}
(1) By Corollary \ref{cor-potgood}, the torsion part of $\Phi_G$
is killed by $e$. Since the kernel of $\alpha$ is killed by $d$,
by Corollary \ref{cor-killd}, and $d$ is prime to $e$, the kernel
of $\alpha$ must be trivial.

(2) Since $t(G)=t(G')=0$ by Corollary \ref{cor-tornochange}, we
know by Proposition \ref{prop-uniftor}  that $G$ and $G'$ do not
admit a subgroup of type $\mathbb{G}_{m,k}$, so that the groups
$\Phi_{G}$ and $\Phi_{G'}$ are finite \cite[10.2.1]{neron}.
 By (1), it suffices to show that $\Phi_G$
and $\Phi_{G'}$ have the same cardinality.

Since $\Phi_G$ and $\Phi_{G'}$ are killed by $e=[L:K]$, and $L$ is
a tame extension of $K$, the values $\phi(G)$ and $\phi(G')$ are
prime to $p$. Therefore, it is enough to prove that
$\phi(G)_q=\phi(G')_q$ for each prime $q\neq p$, where $\phi(G)_q$
denotes the $q$-primary part of $\phi(G)$.

If we put $I'=Gal(K^s/K')$, then
\begin{eqnarray*}
\phi(G)_q&=&|H^1(I,T_qG)_{tors}|
\\ \phi(G')_q&=&|H^1(I',T_qG)_{tors}|
\end{eqnarray*}
by Proposition \ref{prop-compgr}. If we put $I''=Gal(K^s/L)$, then
$I''$ acts trivially on $T_q G$, because $G$ has potential good
reduction (Proposition \ref{prop-goodred}). Since $T_q G$ is
torsion-free, the inflation morphisms $$\begin{array}{c}
H^1(Gal(L/K),T_qG)\rightarrow H^1(I,T_q G)
\\ H^1(Gal(K'L/K'),T_qG)\rightarrow H^1(I',T_q G)
\end{array}$$
are isomorphisms \cite[VII.6.Prop.\,4]{serrelocaux}. Since $[L:K]$
is prime to $[K':K]$, the restriction morphism
$$res:Gal(K'L/K')\rightarrow Gal(L/K)$$ is an isomorphism.
It follows that
$$H^1(I,T_qG)\cong H^1(I',T_q G)$$

(3) As in (2), it suffices to show that $\phi(G)=\phi(G')$. Denote
by $X(G)$ the character module of $G$. It follows from
\cite[7.2.2]{begueri} that
\begin{eqnarray*}
\phi(G)&=&|H^1(Gal(L/K),X(G))|
\\ \phi(G')&=&|H^1(Gal(K'L/K'),X(G))|
\end{eqnarray*}
Since the restriction morphism $res:Gal(K'L/K')\rightarrow
Gal(L/K)$ is an isomorphism, we find
$$\phi(G)=\phi(G')$$
\end{proof}

\begin{cor}\label{cor-potgood3}
Let $G$ be a semi-abelian variety, and let $L$ be a finite
separable extension of $K$ such that $G$ acquires good reduction
over $L$. Let $K'$ be a finite separable extension of $K$ such
that $d=[K':K]$ is prime to $e=[L:K]$. Put $G'=G\times_K K'$, and
consider the morphism
$$\alpha:\Phi_{G}\rightarrow \Phi_{G'}$$ from Corollary
\ref{cor-killd}. Then the following hold:
\begin{enumerate}
\item if $L$ is a tame extension of $K$, then
$$|coker(\alpha)|=d^{t(G)}$$
\item if $G$ is a torus, then
$$|coker(\alpha)|=d^{t(G)}$$
\end{enumerate}
\end{cor}
\begin{proof}
Denote by $T_{sp}$ the maximal split subtorus of $G$. By the proof
of \cite[4.11]{B-X}, the morphism $\alpha$ fits into a morphism of
short exact sequences
$$\begin{CD}
0@>>> \Phi_{T_{sp}} @>>> \Phi_G @>>> \Phi_{G/T_{sp}}@>>> 0
\\ @. @V\beta VV @V\alpha VV @VV\gamma V @.
\\ 0@>>> \Phi_{(T_{sp})'} @>>> \Phi_{G'} @>>> \Phi_{G'/(T_{sp})'}@>>> 0
\end{CD}$$
where $(\cdot)'$ denotes base change to $K'$.
Under the hypotheses of (1) or (2), the morphism $\gamma$ is an
isomorphism, by Propositions \ref{prop-uniftor} and
\ref{prop-potgood}. By Lemma \ref{lemma-linalg}, we get
$$|coker(\alpha)|=|coker(\beta)|$$
It follows from the description of the N\'eron model of
$\mathbb{G}_{m,K}$ in \cite[10.1.5]{neron} that
$$|coker(\beta)|=d^{\mathrm{dim}(T_{sp})}=d^{t(G)}$$
where the equality $\mathrm{dim}(T_{sp})=t(G)$ follows from
 Proposition \ref{prop-uniftor}.
\end{proof}

%
%

The key result of the present paper is the following theorem.
\begin{theorem}\label{thm-main}
Let $A$ be an abelian $K$-variety, with N\'eron model
$\mathcal{A}$, and denote by $t(A)$ the reductive rank of
$\mathcal{A}_s^o$. We denote by $e\in \Z_{>0}$ the degree of the
minimal extension of $K$ where $A$ acquires semi-abelian
reduction.

Let $K'$ be a finite separable extension of $K$, and put
$d=[K':K]$. If $d$ is prime to $e$, then the natural map
$$\Phi_{A}\rightarrow \Phi_{A\times_K K'}$$ is injective. If $A$ is tamely ramified, or $A$ has potential purely
multiplicative reduction, then
$$\phi(A\times_K K')=d^{t(A)}\phi(A)$$
\end{theorem}
\begin{proof}
We fix an embedding of $K'$ in $K^s$. We'll use the notation from
Section \ref{subsec-rigunif}. The short exact sequence of rigid
$K$-groups
$$0\rightarrow M^{an}\rightarrow E^{an}\rightarrow A^{an}\rightarrow 0$$
gives rise to an exact sequence
$$0\rightarrow \mathscr{F}_M\rightarrow \mathscr{F}_E\rightarrow \mathscr{F}_A\rightarrow 0$$
of abelian sheaves on $(\Sp K)_{sm}$ (right exactness of this
sequence follows from the smoothness of $E^{an}\rightarrow
A^{an}$). By Lemma \ref{lemm-exact} we find a commutative diagram
with exact rows
$$\begin{CD}
0@>>> \mathscr{F}_M@>>> \mathscr{F}_E@>>> \mathscr{F}_A @>>> 0
\\ @. @V\tau VV @V\tau VV @VV\tau V @.
\\0@>>> (h_K)_* \mathscr{F}_{M'}@>>> (h_K)_* \mathscr{F}_{E'}@>>> (h_K)_* \mathscr{F}_{A'} @>>> 0
\\ @. @VtrVV @VtrVV @VVtrV @.
\\0@>>> \mathscr{F}_M@>>> \mathscr{F}_E@>>> \mathscr{F}_A @>>> 0
\end{CD}$$
where $\tau$ is the tautological morphism, and $tr$ is the trace
map (Section \ref{subsec-trace}). We know that $tr\circ \tau$ is
multiplication by $d$.

We put $I'=Gal(K^s/K')$. Applying the functor $j_*$, and using
\cite[4.4]{B-X}, we get a commutative diagram with exact rows
$$\begin{CD}
0@>>> M^I@>>> j_*\mathscr{F}_E@>>> j_*\mathscr{F}_A @>\gamma_1 >>
H^1(I,M)
\\ @. @V\alpha_1 VV @V\alpha_2 VV @VV\alpha_3 V @VV\alpha_4 V
\\0@>>> M^{I'}@>>> j_*(h_K)_* \mathscr{F}_{E'}@>>> j_*(h_K)_* \mathscr{F}_{A'} @>\gamma_2 >>
H^1(I',M)
\\ @. @V\beta_1 VV @V\beta_2 VV @VV\beta_3 V @VV\beta_4 V
\\0@>>> M^I@>>> j_*\mathscr{F}_E@>>> j_*\mathscr{F}_A
@>>> H^1(I,M)
\end{CD}$$
where $\beta_i \circ \alpha_i$ is multiplication by $d$, for
$i=1,\ldots,4$. Here we view $M$ as a discrete $I$-module, and the
objects in the second and fifth columns as constant abelian
sheaves on $(\Spf R)_{sm}$.

\bigskip
\textit{Claim 1: $\alpha_1$ is an isomorphism.} Since $A$ acquires
semi-abelian reduction over $L$, the $I$-action on $M$ factors
through $Gal(L/K)$. Since $L$ and $K'$ are linearly disjoint over
$K$, the restriction morphism $$res:Gal(K'L/K')\rightarrow
Gal(L/K)$$ is an isomorphism, and $M^I=M^{I'}$.

\bigskip
\textit{Claim 2: $\alpha_4$ and $\beta_4$ are isomorphisms.}
Consider the commutative diagram
$$\begin{CD}
H^1(I,M)@>\alpha_4>> H^1(I',M)
\\ @Ainf AA @AA inf' A
\\ H^1(Gal(L/K),M)@>\sim>> H^1(Gal(K'L/K'),M)
\end{CD}$$
where the vertical arrows are the inflation morphisms, and the
lower horizontal isomorphism is induced by the isomorphism
$$res:Gal(K'L/K')\rightarrow Gal(L/K)$$
Since $Gal(K^s/L)$ acts trivially on $M$, and $M$ is torsion-free,
the morphisms $inf$ and $inf'$ are isomorphisms
\cite[VII.6.Prop.\,4]{serrelocaux}, so that $\alpha_4$ is an
isomorphism, too. The isomorphism $inf$ shows that $H^1(I,M)$ is
killed by $e$.  Since $d$ is prime to $e$, the composition
$\beta_4\circ \alpha_4$, and hence the morphism $\beta_4$, are
isomorphisms.

\bigskip
Now we pass to component groups. Using \cite[4.12]{B-X} and
Proposition \ref{prop-weilres}, we find a commutative diagram of
constant groups, with exact rows
$$\begin{CD}
0@>>> M^I @>>> \Phi_E@>>> \Phi_A@>\widetilde{\gamma}_1>> H^1(I,M)
\\ @. @V\alpha_1 V\wr V @V\widetilde{\alpha}_2 VV @V\widetilde{\alpha}_3 VV @V\wr V\alpha_4 V
\\ 0@>>> M^{I'}@>>> \Phi_{E'}@>>> \Phi_{A'}@>\widetilde{\gamma}_2>>
H^1(I',M)
\\ @. @. @. @V\widetilde{\beta}_3VV @V\wr V\beta_4 V
\\ @. @. @. \Phi_A@>\widetilde{\gamma}_1>> H^1(I,M)
\end{CD}$$
Since $\widetilde{\alpha}_2$ is injective by Corollary
\ref{cor-potgood3}, a diagram chase shows that
$\widetilde{\alpha}_3$ is injective. Hence, from now on, we may
assume that $A$ is tamely ramified (so that $E$ is tamely
ramified) or $A$ has potential purely multiplicative reduction (so
that $E$ is a torus).

\bigskip
\textit{Claim 3: the isomorphism $\alpha_4$ identifies the images
of $\widetilde{\gamma}_1$ and $\widetilde{\gamma}_2$.} We denote
by $D_A$ and $D'_A$ the subgroups of $\Phi_A$, resp. $\Phi_{A'}$,
consisting of the elements that are killed by a power of $d$.
Since $H^1(I,M)$ and $H^1(I',M)$ are killed by $e$, the morphisms
$\widetilde{\gamma}_1$ and $\widetilde{\gamma}_2$ are trivial on
$D_A$, resp. $D'_A$, and we obtain a commutative diagram
$$\begin{CD}
\Phi_A/D_A@>\overline{\gamma}_1>> H^1(I,M)
\\@V\overline{\alpha}_3VV @V\wr V\alpha_4 V
\\ \Phi_{A'}/D_{A'}@>\overline{\gamma}_2>> H^1(I',M)
\\@V\overline{\beta}_3VV @V\wr V\beta_4 V
\\ \Phi_A/D_A@>\overline{\gamma}_1>> H^1(I,M)
\end{CD}$$
Since $\overline{\beta}_3\circ \overline{\alpha}_3$ is
multiplication by $d$, it is injective, and hence surjective
because $\Phi_A$ is finite. A diagram chase shows that $\alpha_4$
identifies the images of $\overline{\gamma}_1$ and
$\overline{\gamma}_2$.

\bigskip
By Lemma \ref{lemma-linalg}, we know that
 $coker(\widetilde{\alpha}_2)\cong coker(\widetilde{\alpha}_3)$.
Therefore, it suffices to prove that
$$|coker(\widetilde{\alpha}_2)|=d^{t(A)}$$
This follows from Corollary \ref{cor-potgood3}.
\end{proof}


\begin{remarks}
(1) Theorem \ref{thm-main} remains valid if we replace the general
assumption that $K$ is complete, by the assumption that $K$ is
strictly Henselian. It suffices to note that $\phi(A)$ is
invariant under base change to the completion $\widehat{K}$
\cite[7.2.1]{neron}.

(2) The second part of Theorem \ref{thm-main} does not hold for
arbitrary wildly ramified abelian $K$-varieties. A counterexample
is provided in Example \ref{example-elliptic1} below. Note however
that Corollary \ref{cor-potgood3}(3), which is the analog of
Theorem \ref{thm-main} for tori, does not require any tameness
assumption.



(3) It seems plausible that Theorem \ref{thm-main} holds also for
 tamely ramified semi-abelian $K$-varieties $G$, in the
following form: the map
$$\Phi_G\rightarrow \Phi_{G\times_K K'}$$ is injective, and its
cokernel has cardinality $d^{t(G)}$.

(4) From the obvious sequence of maps
$$ T^{an}_{sp} \to T^{an} \to E^{an} \to A^{an} $$
one derives a sequence of component groups
$$ \Phi_{T_{sp}} \to \Phi_T \to \Phi_E \to \Phi_A $$
Taking images in $ \Phi_A $ one obtains, in a canonical way, a filtration
$$ 0=\Sigma_4\subset \Sigma_3 \subset \Sigma_2 \subset \Sigma_1 \subset \Phi_A $$
which was introduced in \cite[\S\,5]{B-X}.

Let us assume that $A$ acquires semi-abelian reduction over a tame
extension of $K$ of degree $e$. Consider a tame extension $K'/K$
of degree $d$, with $d$ prime to $e$. If we put $A'=A\times_K K'$
and denote by $\Sigma'_{\bullet}$ the filtration on $\Phi_{A'}$,
we get a commutative diagram
$$
\begin{CD}
0=\Sigma_4@>>> \Sigma_3 @>>> \Sigma_2 @>>> \Sigma_1 @>>> \Sigma_0=\Phi_A \\
@. @V\gamma_3VV @V\gamma_2 VV @V\gamma_1 VV @VVV
\\0=\Sigma'_4@>>> \Sigma'_3 @>>> \Sigma'_2 @>>> \Sigma'_1 @>>> \Sigma'_0=\Phi_{A'}
\end{CD}
$$
where all maps are injections. For all $ 0 \leq i < j \leq 4 $, we
denote by $$\gamma_{i,j} : \Sigma_i/\Sigma_j \to
\Sigma_i'/\Sigma_j' $$ the map induced by $\gamma_i$. Then one can
show that $ |coker(\gamma_3)| = d^{t(A)} $, and that $
\gamma_{i,j} $ is an isomorphism whenever $ 0 \leq i < j \leq 3 $.
\end{remarks}

\begin{example}\label{example-elliptic1}
In this example we take $R$ to be the Witt vectors $W(k)$, with
$k$ an algebraically closed field of characteristic $2$. Let $C$
be the elliptic $K$-curve with Weierstrass equation
$$ y^2 = x^3 + 2 $$
Recall that $2$ is a uniformizing parameter for $R$.

It is easily computed, using Tate's algorithm, that $C$ has
reduction type $II$ over $R$. Consider $L = K(2^{1/2})$, which is
a wild Kummer extension of $K$ with $[L:K] = 2$. Then $C\times_K
L$ has good reduction.

On the other hand, let $K(3) := K(2^{1/3})$, which is a tame
extension of $K$ with $[K(3):K] = 3$. One checks that $C \times_K
K(3)$ has reduction type $I_0^*$ over $R(3)$. In particular,
$$ 0 = \Phi_C \neq \Phi_{C \times_K K(3)} = (\mathbb{Z}/2\mathbb{Z})^2 $$
We conclude that the second statement in Theorem \ref{thm-main}
does not hold for $C$. Nevertheless, we will see in Lemma
\ref{lemma-wildelliptic} that Theorem \ref{thm-main} holds for all
elliptic curves over $K$ if we replace $e$ by another invariant of
the curve.
\end{example}

\begin{example}\label{example-elliptic2}
Let $C$ be a tamely ramified elliptic curve over $K$ with additive
reduction. In view of Theorem \ref{thm-main}, one might be tempted
to think that $C$ and $C\times_K K'$ have the same reduction type
if $K'$ is a finite extension of $K$ such that $[K':K]$ is prime
to the degree of the minimal extension of $K$ where $C$ acquires
semi-abelian reduction. However, this is not the case.

 Consider, for instance, the elliptic $K$-curve $C$ with
Weierstrass equation
$$ y^2 = x^3 + \pi^4 $$
where $\pi$ is a uniformizer in $R$. We assume that $p > 3$. In
particular, $C$ is tamely ramified.  The minimal extension of $K$
where $C$ acquires semi-abelian reduction is $K(3)$.

Using Tate's algorithm,  one computes that $C$ has reduction type
$IV^*$ over $R$, and that for each $n\in \N'$ such that
 $n \equiv 2$ modulo $6$, the elliptic curve $C \times_K K(n)$ has reduction
type $ IV $ over $R(n)$.
\end{example}

\section{Rationality of the N\'eron component series}\label{sec-ratcomp}

\subsection{Rationality of the component series for tamely
ramified abelian varieties}
\begin{lemma}\label{lemma-res}
Let $P(t)=\sum_{i>0}p_i t^i$ and $Q(t)=\sum_{i>0}q_i t^i$ be
non-zero power series in $\Z[[t]]$ such that $p_i,\,q_i\geq 0$ for
all $i>0$.
 Assume that $P(t)$ and $Q(t)$ converge on the open complex
unit disc $D$, and that they have a pole of order $m_P$, resp.
$m_Q$, at $t=1$. Then $P(t)+Q(t)$ has a pole of order
$\max\{m_P,m_Q\}$ at $t=1$.
\end{lemma}
\begin{proof}
 We may assume that $m_P=m_Q=:m$. It suffices to show that the
residues of $P(t)$ and $Q(t)$ at $t=1$ have the same sign. We
denote these residues by $\rho_P$ and $\rho_Q$, respectively. Let
$(t_n)_{n\geq 0}$ be a series in $]0,1[$ that converges to $1$.
Since the coefficients of $P(t)$ are positive, $P(t)$ takes
positive values on $]0,1[$, so that
$$(-1)^{m}\rho_P=\lim_{n\to \infty}(1-t_n)^m P(t_n)>0$$
Likewise, $(-1)^m\rho_Q>0$.
%
\end{proof}

\begin{lemma}\label{lemma-semipole}
For each element $a$ of $\N$, the series
$$\psi_a(T)=\sum_{d>0}d^aT^d$$ belongs to
$$\Z\left[T,\frac{1}{T-1}\right]$$ It has degree zero if $a=0$ and degree $<0$ else. It has a
pole of order $a+1$ at $T=1$, whose residue equals $(-1)^{a+1}a!$.
\end{lemma}
\begin{proof}
 It suffices to show that $\psi_a(T)$ is a
rational function in $T$ of the form
$$\frac{R_a(T)}{(T-1)^{a+1}}$$ with $R_a(T)\in \Z[T]$,
$\mathrm{deg}\,R_a(T)\leq \max\{a,1\}$ and $R_a(1)=(-1)^{a+1}a!$.
We proceed by induction on $a$. For $a=0$ the result is clear, so
assume that $a>0$ and that the assertion holds for all
$\psi_{a'}(T)$ with $0\leq a'<a$.

Denoting by $\partial_T$ the derivation w.r.t. the variable $T$,
we have \begin{eqnarray*}\psi_a(T)=T\sum_{d>0}(d^{a}T^{d-1})=T
\partial_T\left(\sum_{d>0}d^{a-1}T^d\right)=T\partial
_T\psi_{a-1}(T) \\ =\frac{T((T-1)\partial_T R_{a-1}(T)-a\cdot
R_{a-1}(T))}{(T-1)^{a+1}}\end{eqnarray*} so the result follows
from the induction hypothesis and the fact that $R_0(T)=T$.
%
\end{proof}

\begin{definition}
We define the tame potential toric rank $t_{tame}(A)$ of an
abelian $K$-variety $A$ by
$$t_{tame}(A)=\max \{t(A\times_K K')\,|\,K'\mbox{ a finite tame
extension of }K\}$$
\end{definition}
If $A$ is tamely ramified, then $t_{tame}(A)=t_{pot}(A)$.

\begin{lemma}\label{lemma-tamepot}
Let $A$ be an abelian $K$-variety, and denote by $e$ the degree of
the minimal extension of $K$ where $A$ acquires semi-abelian
reduction. If we denote by $e'$ the prime-to-$p$ part of $e$, then
$$t_{tame}(A)=t(A\times_K K(e'))$$
\end{lemma}
\begin{proof}
By Proposition \ref{prop-torrank}, we may assume that $e'=1$, and
it suffices to show that
$$t(A)=t(A\times_K K')$$ for every finite tame extension $K'$ of
$K$. This follows from Corollary \ref{cor-tornochange}.
\end{proof}

\begin{theorem}\label{thm-ratseries}
Let $A$ be an abelian $K$-variety. Assume that $A$ is tamely
ramified, or that $A$ has potential purely multiplicative
reduction. The N\'eron component series
$$S_{\phi}(A;T)=\sum_{d\in \N'}\phi(A\times_K K(d)) T^d$$ belongs to
$$\mathscr{Z}:=\Z\left[T,\frac{1}{T^j-1}\right]_{j\in \Z_{>0}}$$
It has degree zero if $p=1$ and $A$ has potential good reduction,
and degree $<0$ in all other cases. It has a pole at $T=1$ of
order $t_{tame}(A)+1$.
\end{theorem}
\begin{proof}
For notational convenience, we put $A(d)=A\times_K K(d)$ for every
element $d$ of $\N'$. We denote by $e$ the degree of the minimal
extension of $K$ where $A$ acquires semi-abelian reduction, and by
$\mathscr{D}_e$ the set of divisors of $e$ that belong to $\N'$.
We introduce the series
$$S'_{\phi}(A;T)=\sum_{d\in \N',\,gcd(d,e)=1}\phi(A(d))T^d$$ Then we can write
\begin{eqnarray*}
S_{\phi}(A;T)&=&\sum_{a\in \mathscr{D}_e}\ \sum_{d\in
\N',\,gcd(d,e)=a}\phi(A(d))T^d  \\ &=&\sum_{a\in
\mathscr{D}_e}S'_{\phi}(A(a);T^a)
\end{eqnarray*}
because the degree of the minimal extension of $K(a)$ where $A(a)$
acquires semi-abelian reduction is equal to $e/a$, for each $a$ in
$\mathscr{D}_e$.

 By Lemma \ref{lemma-tamepot}, we have
$$t_{tame}(A)=\max\{t(A(a))\,|\,a\in \mathscr{D}_e\}$$
 In view of Lemma \ref{lemma-res}, it suffices to prove the following claims:
 \begin{enumerate} \item for each $a\in
\mathscr{D}_e$, the series $S'_{\phi}(A(a);T^a)$ belongs to
$\mathscr{Z}$.  It has a pole at $T=1$ of order $t(A(a))+1$. \item
The degree of $S'_{\phi}(A(a);T^a)$ is
\begin{itemize}
\item zero if $p=1$, $a=e$, and $t(A(e))=0$,  \item strictly
negative in all other cases.
\end{itemize}
\end{enumerate}

First, we prove (1). It suffices to consider the case $a=1$. We
denote by $\mathscr{P}_e$ the set of elements in $\{1,\ldots,e\}$
that are prime to $e$. If $p\nmid e$, then for each $b\in
\mathscr{P}_e$, we denote by $n_b$ the smallest element of $b+\N
e$ such that $n_b\notin \N'$. If $p|e$, we put $n_b=0$ for each
$b\in \mathscr{P}_e$. Note that, in any case, $n_b<pe$. We put
$\varepsilon_k=0$ if $p|e$, and $\varepsilon_k=1$ else.

By Theorem \ref{thm-main}, we have
\begin{eqnarray*}
S'_{\phi}(A;T)&=&\sum_{d\in \N',\,gcd(d,e)=1}d^{t(A)}\phi(A) T^d
\\ &=&\phi(A)\cdot \sum_{b\in \mathscr{P}_e}\left(\sum_{q\in \N}
(qe+b)^{t(A)}T^{qe+b}-\varepsilon_k\sum_{r\in
\N}(n_b+epr)^{t(A)}T^{n_b+epr}\right)
\end{eqnarray*}

By Lemma \ref{lemma-semipole}, the series
\begin{equation}\label{eq-series1}\sum_{q\in \N} (qe+b)^{t(A)}T^{qe+b}\end{equation} belongs to
$\mathscr{Z}$, for each $b\in \mathscr{P}_e$. It has a pole of
order $t(A)+1$ at $T=1$, and the residue of this pole equals
$$(-1)^{t(A)+1}e^{-1}(t(A)!)$$ Likewise, the series
\begin{equation}\label{eq-series2}\sum_{r\in \N}(n_b+epr)^{t(A)}T^{n_b+epr}\end{equation} belongs to
$\mathscr{Z}$. It has a pole of order $t(A)+1$ at $T=1$, and the
residue equals
$$(-1)^{t(A)+1}(ep)^{-1}(t(A)!)$$
It follows that $S'_{\phi}(A;T)$ belongs to $\mathscr{Z}$, and
that it has a pole of order $t(A)+1$ at $T=1$.

Now we prove claim (2). If $t(A)>0$, then it follows easily from
Lemma \ref{lemma-semipole} that the series (\ref{eq-series1}) and
(\ref{eq-series2}) have degree $<0$, so that $S'_{\phi}(A(a);T^a)$
has degree $<0$ if $t(A(a))>0$.

If $t(A)=0$, then we find
$$S'_{\phi}(A;T)=\phi(A)\cdot \sum_{b\in
\mathscr{P}_e}\left(\frac{T^b}{1-T^{e}}-\varepsilon_k
\frac{T^{n_b}}{1-T^{ep}}\right)$$ This rational function has
degree $\leq 0$, and its degree is equal to zero iff $e=1$ and
$p=1$. Replacing $A$ by the abelian varieties $A(a)$, for $a\in
\mathscr{D}_e$, we obtain the required result.
\end{proof}

We expect that Theorem \ref{thm-ratseries} holds for all abelian
$K$-varieties. In the following section, we'll prove that this is
true for elliptic curves.

%
%

\subsection{Rationality of the component series for wildly ramified elliptic curves}\label{subsection-wildelliptic}

Assume that $k$ is algebraically closed. Let $C$ be a smooth,
proper,
 geometrically connected $K$-curve of genus $g(C) > 0$,
 and let $ \mathcal{C}/R $ be the minimal sncd-model of $C$
 (i.e., the minimal regular model with strict normal crossings \cite[10.1.8]{liu}).
 We assume that
 either $g(C)>1$, or $C$ is an elliptic curve.

We call an irreducible component $E$ of $\mathcal{C}_s $
\emph{principal} if $g(E) > 0$ or $E$ intersects the other
components in at least three distinct points. Let $e(C)$ be the
least common multiple of the multiplicities of the principal
components of $\mathcal{C}_s$, and let $e(C)'$ be the prime-to-$p$
part of $e(C)$.

If $C$ is tamely ramified (i.e., the wild inertia acts trivially
on the $\ell$-adic cohomology of $C$), then $e(C)=e(C)'$ is equal
to the degree $e$ of the minimal extension of $K$ where $C$
acquires semi-stable reduction, by \cite[7.5]{Ha} (to be precise,
in \cite{Ha} it is assumed that $g(C)>1$, but the proof applies to
elliptic curves as well). On the other hand, if $C$ is the (wildly
ramified) elliptic curve from Example \ref{example-elliptic1},
then $e(C)=6$ while $e=2$, so that not even their prime-to-$p$
parts coincide.


\begin{lemma}\label{lemma-e_p}
Assume that $k$ is algebraically closed. Let $C$ be an elliptic
curve over $K$, and let $a$ be a divisor of $e(C)'$.
 Then $ e(C \times_K K(a))' = e(C)'/a $.
\end{lemma}
\begin{proof}
If $C$ is tamely ramified, this follows from the equality
$e=e(C)$. Hence, we may assume that $C$ is wildly ramified.
 Considering the Kodaira classification and applying Saito's criterion for wild ramification \cite[3.11]{saito},
 $e(C)' > 1$  only occurs
when $ p $ is either $ 2 $ or $ 3 $ and the reduction type of $C$
is either $ II $ or $ II^* $. We will give a detailed argument
when $p = 2$ and $C$ has reduction type $II$, the remaining cases
follow in a similar fashion. In this case, we have $e(C)'=3$. For
$a=1$ there is nothing to prove, so we may assume that $a=3$.

We will use the computations and results from \cite{Ha}. It is a
slight problem that there exists a pair of intersecting components
of $\mathcal{C}_s$ that both have multiplicities divisible by $p$,
since locally at such intersection points, the methods of
\cite{Ha} don't apply. However, because of the very limited
possibilities of degeneration types for elliptic curves, we can
get around this with some ad hoc arguments.

The special fiber $\mathcal{C}_s$ is of the form
$$ \mathcal{C}_s =  E_1 + 2 E_2 + 3 E_3 + 6 E_4 $$
where $E_4$ meets each other component transversally in a unique
point, and the other components are pairwise disjoint. We denote
by  $ \mathcal{D} $ the normalization of $ \mathcal{C} \times_R
R(3) $. It follows from \cite[2.1+2.9+6.3]{Ha} that
$$ \mathcal{D}_s = F_1 + 2 F_2+ F_3^1 + F_3^2 +  F_3^3 + 2 F_4   $$
where $ F_i $ dominates $E_i$, for $i=1,2,4$, and $F^j_3$
dominates $E_3$, for $j=1,2,3$. Moreover, $F_4$ intersects $ F_3^j
$ and $F_1$ transversally at a unique point, and $F_2\cap F_4\neq
\emptyset$. It follows from \cite[4.3]{Ha} that $ \mathcal{D} $ is
regular outside of $ F_2 \cap F_4 $, and from \cite[2.2+2.9]{Ha}
that all the components of $\mathcal{D}_s$ are smooth, except
possibly for $F_2$ and $F_4$ at the points where they intersect.

Let $$ \rho : \mathcal{C}(3) \to \mathcal{D} $$ be the minimal
sncd-desingularization.  It follows from what we've said above
that $ \rho $ is an isomorphism above $ \mathcal{D} \setminus \{
F_2 \cap F_4 \} $. Let $$ \tau : \mathcal{C}(3) \to
\mathcal{C}(3)_{min} $$ be the canonical morphism to the minimal
sncd-model $\mathcal{C}(3)_{min}$ of $C\times_K K(3)$. It is
obvious that $ \tau $ is an open immersion when restricted to $
\rho^{-1}( \mathcal{D} \setminus \{ F_2 \cap F_4 \} ) $. It
follows that the special fiber of $ \mathcal{C}(3)_{min} $
contains a component with multiplicity $2$, meeting four reduced
components each in a unique point. The only possibility is then
reduction type $I_0^*$.

These are the results in the other cases:
\begin{itemize}
\item if $p=2$ and $C$ has type $II^*$, then $e(C)'=3$ and
$C\times_K K(3)$ has type $I_0^*$, \item if $p=3$ and $C$ has type
$II$, then $e(C)'=2$ and $C\times_K K(2)$ has type $IV$, \item if
$p=3$ and $C$ has type $II^*$, then $e(C)'=2$ and $C\times_K K(2)$
has type $IV^*$.
\end{itemize}
\end{proof}

\begin{lemma}\label{lemma-wildelliptic}
Assume that $k$ is algebraically closed. Let $C$ be an elliptic
curve that does not have multiplicative reduction. For every
finite tame extension $ K'/K $ whose degree is prime to $e(C)$, we
have
$$ \phi(C \times_K K') = \phi(C) $$
\end{lemma}
\begin{proof}
If $C$ is tamely ramified, this follows from Theorem
\ref{thm-main}. We give an alternative proof that is valid also in
the wild case. We may assume that $C$ has additive reduction.
Looking at the Kodaira reduction table, one sees that the special
fiber of the minimal $sncd$-model $\mathcal{C}$ of $C$ contains a
principal component, that all the principal components of
$\mathcal{C}_s$ have the same multiplicity $m$, and that this
multiplicity $m$ determines $\phi(C)$ (the principal component is
unique except in the case where $C$ has reduction type $I_n^*$
with $n>0$).

Explicitly, we have
$$m=\left\{\begin{array}{ccl}
1 & \mbox{if }$C$ \mbox{ has type}& I_0
\\ 2 & & I_{n}^*,\, n\geq 0
\\ 6& & II\mbox{ or }II^*
\\ 4 & & III\mbox{ or }III^*
\\ 3& & IV\mbox{ or }IV^*
\end{array}\right.$$
Reasoning as in the proof of Lemma \ref{lemma-e_p}, one sees that
each principal component of $\mathcal{C}_s$ gives rise to a
principal component with the same multiplicity in the minimal
$sncd$-model of $C\times_K K'$. Hence,
$$ \phi(C \times_K K') = \phi(C) $$
\end{proof}

\begin{prop}\label{proposition-wildelliptic}
Assume that $k$ is algebraically closed. Let $C$ be an elliptic
curve over $K$. The component series
$$S_{\phi}(C;T)=\sum_{d\in \N'}\phi(C \times_K K(d)) T^d$$ belongs to
$$\mathscr{Z}:=\Z\left[\frac{T^j}{1-T^j}\right]_{j\in \Z_{>0}}$$
It has a pole at $T=1$ of order $t_{tame}(C)+1$. It has degree
zero if $p=1$ and $C$ has potential good reduction, and degree
$<0$ in all other cases.
\end{prop}
\begin{proof}
If $C$ is tamely ramified, then this follows from Theorem
\ref{thm-ratseries}. Hence, we may assume that $C$ is wildly
ramified. In this case, $p>1$ and $t_{tame}(C)=0$. For notational
convenience, we put $C(d)=C\times_K K(d)$ for each $d\in \N'$.

We can write
$$ S_{\phi}(C;T) = \sum_{a|e(C)'}\  \sum_{d \in \mathbb{N}', gcd(d,e(C)') = a} \phi(C (d))T^d $$
 By Lemmas \ref{lemma-e_p} and \ref{lemma-wildelliptic}, this expression equals
$$ \sum_{a|e(C)'} \left( \phi(C \times_K K(a)) \cdot \sum_{d \in \mathbb{N}', gcd(d,e(C(a))') = 1} T^{ad} \right) $$
Direct computation shows that this series belongs to
$\mathscr{Z}$, that it has a pole of order one at $T=1$, and that
it has strictly negative degree.
\end{proof}


\begin{thebibliography}{10}
\bibitem{sga3.2}
{\em Sch\'emas en groupes. {II}: Groupes de type multiplicatif, et
structure
  des sch\'emas en groupes g\'en\'eraux}.
\newblock S\'eminaire de G\'eom\'etrie Alg\'ebrique du Bois Marie 1962/64 (SGA
  3). Dirig\'e par M. Demazure et A. Grothendieck. Lecture Notes in
  Mathematics, Vol. 152. Springer-Verlag, Berlin, 1970.

\bibitem{sga7a}
{\em Groupes de monodromie en g\'eom\'etrie alg\'ebrique. {I}}.
\newblock Springer-Verlag, Berlin, 1972.
\newblock S\'eminaire de G\'eom\'etrie Alg\'ebrique du Bois-Marie 1967--1969
  (SGA 7 {I}), Dirig\'e par A. Grothendieck. Avec la collaboration de M.
  Raynaud et D. S. Rim, Lecture Notes in Mathematics, Vol. 288.

\bibitem{ananth}
S.~Anantharaman. \newblock{Sch\'emas en groupes, espaces
homog\`enes et espaces alg\'ebriques sur une base de dimension 1.}
\newblock{\em M\'emoires de la Soci\'et\'e Math\'ematique de France},
33:5--79, 1973.

\bibitem{begueri}
L.~Begueri.
\newblock {Dualit\'e sur un corps local \`a corps r\'esiduel alg\'ebriquement
  clos}.
\newblock {\em M{\'e}m. Soc. Math. Fr., Nouv. S{\'e}r. 4}, 121, 1980.

\bibitem{bertapelle}
A.~Bertapelle. \newblock{Formal {N}\'eron models and {W}eil
restriction.} \newblock {\em Math. Ann.}, 316(3):437--463, 2000.


\bibitem{neron}
S.~Bosch, W.~{L\"u}tkebohmert, and M.~Raynaud.
\newblock {\em {N\'eron models}}, volume~21 of
 {\em Ergebnisse der Mathematik und ihrer Grenzgebiete}. \newblock Springer-Verlag, 1990.

\bibitem{formner}
S.~Bosch and K.~Schl{\"o}ter.
\newblock N\'eron models in the setting of formal and rigid geometry.
\newblock {\em Math. Ann.}, 301(2):339--362, 1995.

\bibitem{B-X}
S.~Bosch and X.~Xarles.
\newblock {Component groups of N{\'e}ron models via rigid uniformization}.
\newblock {\em Math. Ann.}, 306:459--486, 1996.

\bibitem{chai}
C.L. Chai.
\newblock {N\'eron models for semiabelian varieties: congruence and change of
  base field}.
\newblock {\em Asian J. Math.}, 4(4):715--736, 2000.

\bibitem{conrad-chevalley}
B.~Conrad.
\newblock {A modern proof of Chevalley's theorem on algebraic groups.}
\newblock {\em J. Ramanujan Math. Soc.}, 17(1):1--18, 2002.

\bibitem{ELL}
B.~Edixhoven, Q.~Liu and D. Lorenzini.
\newblock The $p$-part of
the group of components of a N\'eron model.
\newblock{\em J. Algebr. Geom.}, 5(4):801--813, 1996.

\bibitem{Ha}
L.H.~Halle
\newblock{Stable reduction of curves and tame ramification.}
\newblock{To appear in Math.~Zeit.}, arXiv:0711.0896

\bibitem{HaNi}
L.H.~Halle and J.~Nicaise.
\newblock{Motivic zeta functions of abelian varieties, and the
monodromy conjecture.}
\newblock{\em preprint}, arXiv:0902.3755v3.

\bibitem{HaNi-new}
L.H.~Halle and J.~Nicaise.
\newblock{Motivic zeta functions of abelian varieties, and the
monodromy conjecture.}
\newblock{\em in preparation}.

\bibitem{liu}
Q.~Liu.
\newblock {\em Algebraic geometry and arithmetic curves}, volume~6 of {\em
  Oxford Graduate Texts in Mathematics}.
\newblock Oxford University Press, 2002.

\bibitem{N-X}
E.~Nart and X.~Xarles. \newblock{Additive reduction of algebraic
tori.} \newblock{\em Arch. Math.}, 57:460--466, 1991.



\bibitem{saito}
T.~Saito.
\newblock {Vanishing cycles and geometry of curves over a discrete valuation
  ring.}
\newblock {\em Am. J. Math.}, 109:1043--1085, 1987.


\bibitem{serrelocaux}
J.-P. Serre.
\newblock {\em {Corps locaux}}.
\newblock {Paris: Hermann \& Cie}, 1962.


\end{thebibliography}
\end{document}